\documentclass{article}
\usepackage[T1]{fontenc}
\usepackage{hyperref}
\hypersetup{colorlinks}
\usepackage{iftex}

\newtheorem{theorem}{Theorem}[section]
\newtheorem{lemma}[theorem]{Lemma}
\newtheorem{corollary}[theorem]{Corollary}
\newtheorem{proposition}[theorem]{Proposition}
\newtheorem{proof}[theorem]{Proof}
\newtheorem{keywords}[theorem]{Keywords}

\newtheorem{definition}[theorem]{Definition}

\title{The Source Of Primeness Of Rings}
\author{Didem Ye\c{s}il\textsuperscript{1}\thanks{Didem YEŞİL. Email: dyesil@comu.edu.tr} \,and D\.{i}dem Karalarl\i o\u{g}lu Camc\i\textsuperscript{2}\thanks{Didem Karalarlıoğlu Camcı. Email: didemk@comu.edu.tr}\thanks{%
\url{https://github.com/latex3/iftex}}}

\date{\csname ver@iftex.sty\endcsname}

\begin{document}

\maketitle
\tableofcontents

\begin{abstract}

Let $R$ be an associative ring. We define a subset $S_{R}^{a}$, where $a\in R$ of $R$ as $S_{R}^{a}=\{b\in R \mid aRb=(0)\}$. Then, the set $P_{R} = \bigcap_{a\in R} S_{R}^{a}$ call it the source of primeness of $R$. We first examine some basic properties of the subset $P_{R}$ in any ring $R$, and properties of idempotent elements, nilpotent elements, zero divisor elements and identity elements. And we investigated the properties of the elements of the set source of primeness with the help of these elements.

\end{abstract}

\begin{keywords}
	Prime ring; semiprime ring; source of primeness; source of semiprimeness
\end{keywords}

\section{Introduction}
Our main aim in this study is to describe three types of rings, which, as far as we know, have not been included in the literature before. These definitions are motivated by original existing concepts in ring theory and can be viewed as generalizations of reduced rings, domain and division rings, respectively (see Definition \ref{def1}). To define these new concepts of rings, we will first introduce a particular set, which is the subset of the ring, and which we call the source of primeness of the ring in question in the light of article \cite{Didem} . Before getting down into the subject matter, let's outline the terminology we will use throughout the article.

An element in a unitary ring with a right (\emph{resp.} left) multiplicative
inverse will be called a right (\emph{resp.} left) unit, and accordingly it
will be meant by a unit a two-sided unit. An element $a$ of a ring $R$ is
called a right (\emph{resp.} left) zero-divisor if there exists a nonzero
element $b\in R$ such that $ba=0$ (\emph{resp.} $ab=0$). An element which is
neither a left nor a right zero-divisor is called a nonzero-divisor. A ring
which has no nonzero right or left zero-divisors is called a domain, and a
ring whose nonzero elements are all units is called a division ring. An
element $a$ of a ring $R$ is called a nilpotent element of index $n$ if $n$
is the least positive integer such that $a^{n}=0$. A ring with no nonzero
nilpotent elements is called a reduced ring. An idempotent element $e=e^{2}\in R$ is called central if it commutes with every element of $R$,
that is to say $e$ is contained in the center of $R$.

Following \cite{McCoy}, we define a ring $R$ to be a prime (\emph{resp}.
semiprime) ring if the zero ideal is a prime (\emph{resp}. semiprime) ideal
of $R$. Equivalently, $R$ is called a prime ring if $aRb=(0)$ with $a,b\in R$
implies $a=0$ or $b=0$; and $R$ is called a semiprime ring if $aRa=(0)$ with 
$a\in R$ implies $a=0$. As it is well-known, each ring is isomorphic to the subdirect sum of the prime rings.

It is now convenient to introduce our main instrument what we focus our
attention on throughout the paper. For a ring $R$, we define the subset $S_{R}^{a}$ of $R$ as such;
$$S_{R}^{a}=\{b\in R \mid aRb=(0)\}.$$ 
And we call the following set as the source of primeness of $R$
$$P_{R} = \bigcap_{a\in R} S_{R}^{a}.$$
It is always a nonempty set as it contains $0$. At one extreme, $P_{R}$ may consist only of $0$ in which case we say $P_{R}$ is trivial, and at another extreme, $P_{R}$ may contain whole of $R$. The triviality of $P_{R}$ is only possible when $R$ is a prime ring. Putting these both aside, our
general concern will be substantially the cases between these two extremes.
A rigorous reader should have already noticed that $P_{R}$ is always
contained in the prime radical of $R$, since $ P_{R}\subseteq S_{R}$. 

So by looking more closely at the elements of $R-P_{R}$, we can say that the set $P_{R}$ is not that large of subset to miss out the chance to examine the structure of $R$. Let's say a few words about the name we suggested for $P_{R}$. For $P_{R}$ we prefer the name `` the source of primeness of $R$ " 
because each element in $R-S_{R}$ behaves like a nonzero element does in any prime ring: $RRb\neq(0)$ for every $b\in R-P_{R}$, explains where the `` primeness " part comes from.

In section $2$, we investigate basic algebraic properties of $P_{R}$ for any
ring $R$, and most of the results in this section will be of elementary
type. For instance, we shall show that $P_{R}$ is an ideal of $R$. Another result
worth mentioning here is that the source of semiprimeness is preserved under
ring isomorphisms.

\section{Results}
\begin{definition}\label{def1}
	Let $R$ be a ring, $\emptyset\neq A\subseteq R$ and $a\in R$. We define $S_{R}^{a}(A)$ as follows:
	$$S_{R}^{a}(A)= \{b\in R \,\,\mid \,\, aAb=(0)\}. $$	
	$P_{R}(A)= \bigcap_{a\in R} S_{R}^{a}(A)$ is called the source of primeness of the subset $A$ in $R$. We write $S_{R}^{a}$ instead of  $S_{R}^{a}(R)$. In particular, we can similarly define the source of primeness of the semigroup $R$ as follows: $$P_{R} = \bigcap_{a\in R} S_{R}^{a}.$$
\end{definition}
First, let's talk about some inferences that are easy to see but will help understand the set.
\begin{enumerate}
	\item Since $R$ be a ring, we obtain $aA0=(0)$ for all $a\in R$.  From this fact $S_{R}^{a}(A)\neq \emptyset$, $\forall a\in R$.Hence $P_{R}=\bigcap_{a\in R} S_{R}^{a}\neq \emptyset$.
	\item  $S_{R}^{0}(A) =R$.
	\item $S_{A}^{a} \subseteq  S_{R}^{a}(A)$. If $b\in S_{A}^{a}$, then $b\in A$ such that $aAb=(0)$. Since $A\subseteq R$, we have $b\in R$ and $aAb=(0)$. This means that $b\in S_{R}^{a}(A)$.
\end{enumerate}
If $x\in P_{R}(A)$, then $aAx=(0)$, for all  $a\in R$. Hence, $RAx=(0)$. Therefore, $P_{R}(A)= \{x\in R \,\,:\,\, RAx=(0)\}.$
\begin{theorem}
	Let $R$ be a ring and $\emptyset\neq A, B\subseteq R$. Then, $P_{R\times R}(A\times B)=  P_{R}(A)\times P_{R}(B)$.
	
\end{theorem}
\begin{proof}
	$P_{R\times R}(A\times B) = \{(x,y)\in R\times R \mid (R\times R)(A\times B)(x,y)= (0,0)\}.$
	\begin{eqnarray}
		(x,y)\in P_{R\times R}(A\times B)
		&\iff& (R\times R)(A\times B)(x,y)= (0,0)\\
		&\iff& (r_{1},r_{2})(a,b)(x,y)= (0,0)\\
		&\iff& (r_{1}ax,r_{2}by)=(0,0),\,\, \forall r_{1},r_{2}\in R, \,\, \forall a\in A \,\, \forall b\in B\\
		&\iff& (RAx)=(0), \,\, (RBy)=(0)\\
		&\iff& x\in P_{R}(A), \,\, y\in P_{R}(B)\\
		&\iff& (x,y)\in P_{R}(A)\times P_{R}(B)
	\end{eqnarray}
	
	Namely, $P_{R\times R}(A\times B)=  P_{R}(A)\times P_{R}(B)$.
\end{proof}
\begin{theorem}
	If $1_{R}\in R$ and $R$ is commutative, then $P_{R}\subseteq \{x\in R \mid x^{2}=0\}.$
\end{theorem}
\begin{proof}
	Let $K=\{x\in R \mid x^{2}=0\}.$
	
	If $x\in P_{R}$, then $RRx=(0)$. We have $xRx=(0)$ for $x\in R$. Since $R$ is commutative and $1_{R}\in R$, $x\in K$ is satisfied. So, $P_{R}\subseteq K$. 
	
\end{proof}
\begin{lemma}
	If $R$ is a prime ring, then $P_{R} =\{0\}.$
\end{lemma}
\begin{proof}
	If $x\in P_{R}$, then $RRx=(0)$. Since $R$ is a prime ring, $x=0$. Hence, $P_{R} =\{0\}.$
\end{proof}

\begin{lemma}
	Let $R$ be a ring and $A,B$ are nonempty subsets of $R$. Then the following conditions hold:
	\begin{enumerate}
		\item If $A\subseteq B$, then $P_{R}(B) \subseteq P_{R}(A)$. In particular, $P_{R}\subseteq P_{R}(A)$.
		\item If $A$ is subring of $R$, then $A\cap P_{R}(A)\subseteq P_{A}$.
	\end{enumerate}
\end{lemma}
\begin{proof}
	\begin{enumerate}
		\item Let $x\in P_{S}(B)$. We have $x\in \bigcap_{a\in S} S_{S}^{a}(B)$ and $aBx=(0)$, $\forall a\in S$ follows from the definition. Since $A\subseteq B$, we write $aAx=(0)$ for all $a\in R$. This means that $x\in S_{R}^{a}(A)$ for all $a\in R$. Hence we get $x\in \bigcap_{a\in R}S_{R}^{a}(A)$ and $x\in P_{R}(A)$. This gives up $P_{R}(B)\subseteq P_{R}(A)$. Specially $P_{R}\subseteq P_{R}(A)$ is satisfied for $A\subseteq R$.
		\item Let $x\in A\cap P_{R}(A)$. Then, $x\in A$ and $x\in P_{R}(A)$. Hence we get $x\in A$ and $x\in\bigcap_{a\in  R} S_{R}^{a}(A)$, and so, $x\in S_{R}^{a}(A)$ for all $a \in R$. Using $x\in A$, $x\in S_{A}^{a}$ for all $a\in A$. This expression gives us $x\in \bigcap_{a\in R}S_{A}^{a}= P_{A}$. So, we obtain $A\cap P_{R}(A)\subseteq P_{A}$.
	\end{enumerate}
\end{proof}
It is well known that every prime ring is a semiprime ring. Now let's look at the relationship between the source of primeness and the source of semiprimeness.
\begin{theorem}
	Let $R$ be a ring, $\emptyset\neq A\subseteq R$ and $a\in R$. Then $P_{R}(A)\subseteq S_{R}(A)$.
\end{theorem}
\begin{proof}
	If $b\in P_{R}(A)$, then $b\in\bigcap_{a\in R} S_{R}^{a}$. 
	In particular, $b\in S_{R}^{b}.$ 
	Therefore, $bAb=(0).$ Hence, $b\in S_{R}(A).$ 
\end{proof}
\begin{proposition}\label{prop2}
	Let $R$ be a ring, $a\in R$ and $I$ be a nonempty subset of $R$. In this case, the following features are provided.
	\begin{enumerate}
		\item $S_{R}^{a}(I)$ is a right ideal of $R$.
		\item If $I$ is a right ideal of $R$, then $S_{R}^{a}(I)$ is a left ideal of $R$.
		\item If $I$ is a right ideal of $R$, then $S_{R}^{a}(I)$ is a ideal of $R$. 
	\end{enumerate}
\end{proposition}
\begin{proof}
	\begin{enumerate}
		\item \label{1} Let $x,y\in S_{R}^{a}(I)$. Then, $aIx=aIy=(0)$ for $a\in R$. From here $aI(x-y)=aIx-aIy=(0)$, we obtain $x-y\in S_{R}^{a}(I)$. Besides that, we have $aI(xr)=(aIx)r=(0)$ for $r\in R$. Thus, we get $xr\in S_{R}^{a}(I)$. Hereby, $S_{R}^{a}(I)$ is a right ideal of $R$.
		\item \label{2} Let $I$ is a right ideal of $R$ and $x\in S_{R}^{a}(I)$, $r\in R$. Then, we get $aI(rx)= a(Ir)x\subseteq aIx=(0)$. Hence, we have $rx\in S_{R}^{a}(I)$ and $S_{R}^{a}(I)$ is a left ideal of $R$.
		\item We can easily see that if $I$ is a right ideal of $R$, then $S_{R}^{a}(I)$ is an ideal of $R$ from \ref{1} and \ref{2}. 
	\end{enumerate}
\end{proof}
\begin{lemma}
	Let $R$ be a ring and $I$ be a nonempty subset of $R$. If $I$ is a right ideal of $R$, then $P_{R}(I)$ is an ideal of $R$. 
\end{lemma}
\begin{proof}
	Let $I$ be a right ideal of $R$ and $x,y \in P_{R}(I)$. Therefore, $x,y\in \bigcap_{a\in R}S_{R}^{a}(I)$ and $x,y\in S_{R}^{a}(I)$ for all $a\in R$. Then, $aIx=aIy=(0)$ for all $a\in R$. Since $S_{R}^{a}(I)$ is an ideal of $R$  for all $a\in R$ from Proposition \ref{prop2}, we write $x-y, xr, rx \in S_{R}^{a}(I)$ for all $a,r\in R$. Consequently, we get $x-y, xr, rx \in \bigcap_{a\in R}S_{R}^{a}(I) =P_{R}(I)$. For this reason, $P_{R}(I)$ is an ideal of $R$.
\end{proof}
Let us now examine the elements of set of the source of primeness $P_{R}(A)$ given as the intersection of sets $S_{R}^{a}(A)$. For this, let's consider the set $B=\{b\in R \mid aAb=(0), \,\,\,\forall a\in R\}$. If $b\in B$, then $aAb=(0)$ for all $a\in R$. Hence we get $b\in S_{R}^{a}(A)$ for all $a\in R$. This gives us $b\in P_{R}(A)$. Conversely, if $x\in P_{R}(A)$, then $aAx=(0)$ for all $a\in R$. So, we obtain $x\in B$. 

In the following lemma, if the ring $R$ is prime, its relation to the set the source of primeness examined.
\begin{lemma}
	Let $R$ be a ring. Thus the followings are provided.
	\begin{enumerate}
		\item If $R$ is prime ring, then $P_{R}=\{0\}$.
		\item The source of primeness $P_{R}$ is contained by every prime ideal of the $R$. 
	\end{enumerate}
\end{lemma}
\begin{proof}
	\begin{enumerate}
		\item Let $R$ be a ring and $x\in P_{R}$. From definition of the set $P_{R}$, we have $RRx=(0)$. Since $R$ is prime ring, we obtained $x=0$. Namely, $P_{R}=\{0\}$.
		\item Let $P$ is prime ideal in $R$. If $x\in P_{R}$, then $RRx=(0)\subseteq P$ by the definition of $P_{R}$. Since $P$ is prime, we get $x\in P$. Hence, we get $P_{R}\subseteq P$. This gives us that every prime ideal of $R$ includes $P_{R}$.  
	\end{enumerate}
\end{proof}
We know that the properties of idempotent, nilpotent, and zero-divisor elements in a ring are also related to the primality of that ring. Now let's examine the relationship between the source of primeness and these special elements.
\begin{lemma}\label{lem5}
	Let $R$ be a ring. Then the following holds.
	\begin{enumerate}
		\item If every element in $R$ is idempotent element, then $P_{R}=\{0\}$.
		\item If $a\in P_{R}$, then $a$ is a zero divisor element.
		\item If $R$ has identity element, then $P_{R}=\{0\}$.
	\end{enumerate}
\end{lemma}
\begin{proof}
	\begin{enumerate}
		\item Let every element in $R$ be idempotent element and $a\in P_{R}$. Since $RRa=(0)$, $aaa=0$. Moreover, $a$ is an idempotent element, we have $a=0$. Then, $P_{R}=\{0\}$.
		\item If $0\neq a\in P_{R}$, then $RRa=(0)$. We have $axa=0$ for all $x\in R$. Namely, $a(xa)=0$. If $ax\neq 0$, then $x$ is a left zero divisor element. Also, since $R\neq (0)$, if $ax=0$  for all $a\in R$, then $xa=0$ for some $x\neq 0$. That means $a$ is a left zero divisor element. Hence, we can see $a$ is a right zero divisor in the same way. Thus, we get that $a$ is a zero divisor element.
		\item Let $R$ has identity element $1_{R}$ and $a\in P_{R}$. Then, we have $RRa=(0)$. In particular, we write $1_{R}1_{R}a=0$. Then, we obtain $P_{R}=\{0\}$.
	\end{enumerate}
\end{proof}
From Lemma \ref{lem5}, it is easy to see that following corollary.

\begin{corollary}
	For any $R$ ring the following is always true.
	\begin{enumerate}
		\item There is no idempotent element other than zero in $P_{R}$.
		\item Every element in $P_{R}$ is nilpotent element.
		\item Every element in $P_{R}$ is zero divisor element.
		
	\end{enumerate}
\end{corollary}
\begin{proposition}
	Let $R$ and $T$ be two rings and $f: R\rightarrow T$ a ring homomorphism. So, $f(P_{R})\subseteq P_{f(R)}$. If $f$ is injective, then $f(P_{R})= P_{f(R)}$.
\end{proposition}
\begin{proof}
	Let $x\in f(P_{R})$. In this case, there is $a\in P_{R}$ such that $x=f(a)$. Namely, $RRa=(0)$. Since $$(0)=f(RRa)=f(R)f(R)f(a),$$ we get $f(a)\in P_{f(R)}$ and so $x\in P_{f(R)}$. Thence, $f(P_{R})\subseteq P_{f(R)}$.
	
	Now let's take $y\in P_{f(R)}$. From the set definition, $f(R)f(R)y=(0)$ for $y\in f(R)$. Since $y\in f(R)$, we have $y=f(r)$ for $r\in R$. Hence, $f(R)f(R)f(r)=(0)$ and since $f$ is a homomorphism $f(RRr)=(0)$ is satisfied. Using $f$ is injective, we obtain $RRr=(0)$ for $r\in R$. So, $r\in P_{R}$. From here, $y=f(r)\in f(P_{R})$ is obtained. Thus, $P_{f(R)}\subseteq f(P_{R})$.
\end{proof}

\section{Conclusion}

By adding the definition of the source of the primeness to the literature, we investigated the relationships between the source of the primeness and the prime ring. However, we have given some results of the source set of primeness. In addition, the relationships between the inverse and completely regular elements. Apart from that, we gave the results of different theorems and showed these theorems on examples.

\section{Data Availability}
The data used to support the findings of this study are available from the corresponding author upon request.

\section{Author Contributions}
The YEŞİL posed the problem and obtained 80 percent of the data contained in the article. She is also responsible for uploading the manuscript to the journal and making corrections. Camci checked the data obtained by the first author. Also, she has achieved the some results.  They all read and approved the last version of the manuscript.

\section{Conflicts of Interest}
The authors declare no conflict of interest.



\begin{thebibliography}{References}
	\bibitem[Aydın et al.(2018)Aydın, Demir, Camcı]{Didem}
	Aydın, N., Demir, Ç., Camcı, D.~K. (2018). The source of semiprimeness of rings. \emph{Communications of the Korean Mathematical Society}, \emph{33},  1083-1096.
	\bibitem[McCoy (1964)]{McCoy} N. H. McCoy, \emph{The Theory of Rings}. The Macmillan Co.,
	New York; Collier-Macmillan Ltd., London, 1964. x+161 pp.
	
	\bibitem[LAM (1991)]{Lam1} T. Y. Lam, \emph{A first course in noncommutative rings}.
	Graduate Texts in Mathematics, 131. Springer-Verlag, New York, 1991. xvi+397
	pp.
	
	\bibitem[Lam(1999)]{Lam2} T. Y. Lam, \emph{Lectures on modules and rings}. Graduate
	Texts in Mathematics, 189. Springer-Verlag, New York, 1999. xxiv+557 pp.
	
	
\end{thebibliography}
\end{document}